\documentclass{amsart}
\usepackage{amsmath,amssymb,amsxtra,anysize,adjustbox,xcolor,combelow,hyperref}
\usepackage{framed}

\usepackage{tikz-cd,comment}
\usetikzlibrary{calc, babel}%
\usetikzlibrary{shapes}%
\usetikzlibrary{patterns}%
\usetikzlibrary{positioning}%
\usetikzlibrary{arrows.meta}
\usetikzlibrary{knots}
\usetikzlibrary{decorations.markings}
\usetikzlibrary{fit}
\usetikzlibrary{hobby}
\tikzset{curve/.style={settings={#1},to path={(\tikztostart)
    .. controls ($(\tikztostart)!\pv{pos}!(\tikztotarget)!\pv{height}!270:(\tikztotarget)$)
    and ($(\tikztostart)!1-\pv{pos}!(\tikztotarget)!\pv{height}!270:(\tikztotarget)$)
    .. (\tikztotarget)\tikztonodes}},
    settings/.code={\tikzset{quiver/.cd,#1}
        \def\pv##1{\pgfkeysvalueof{/tikz/quiver/##1}}},
    quiver/.cd,pos/.initial=0.35,height/.initial=0}

\newcommand{\seed}{\Sigma}
\newcommand{\excluster}[1]{\widetilde{\mathbf{#1}}}

\newcommand{\mut}{\mathrm{mut}}
\newcommand{\exmatrix}[1]{\widetilde{#1}}
\newcommand{\var}{\mathcal{A}}
\newcommand{\deep}{\mathcal{D}}
\newcommand{\manifold}{\mathcal{M}}
\newcommand{\stab}{\mathcal{S}}
\newcommand{\Stab}{\mathrm{Stab}}

\newcommand{\C}{\mathbb{C}}
\newcommand{\Z}{\mathbb{Z}}

\newcommand{\Spec}{\mathrm{Spec}}

\newcommand{\seeds}{\mathsf{Seeds}}

\usepackage{enumitem}

\usepackage{mathtools}

\numberwithin{equation}{section}

\newtheorem{theorem}{Theorem}[section]
\newtheorem{proposition}[theorem]{Proposition}
\newtheorem{corollary}[theorem]{Corollary}
\newtheorem{lemma}[theorem]{Lemma}

\theoremstyle{definition}
\newtheorem{remark}[theorem]{Remark}

\newtheorem{definition}[theorem]{Definition}
\newtheorem{example}[theorem]{Example}

\newcommand{\dil}{\mathrm{Dil}}

\newcommand{\ambient}{\mathcal{F}}

\title{Mysterious points in keys but not trees}
\author{Scott Neville}
\address{Laboratoire d'Algèbre, de Combinatoire et d’informatique Mathématique (LACIM)\\ Université du Québec à Montréal\\ Montréal, Québec, Canada }
\email{nevilles@umich.edu}

\author{Jos\'e Simental}
\address{Instituto de Matemáticas, Universidad Nacional Autónoma de México \\ Ciudad Universitaria, CDMX, México }
\email{simental@im.unam.mx}


\begin{document}

\begin{abstract}
The deep locus of a cluster variety is defined to be the set of its points that do not belong to any cluster torus. We show that, if the cluster variety has a seed whose mutable part is a tree without multiple edges, then the deep locus can be characterized as the set of points whose stabilizer under a certain group action is nontrivial. Deep points without a stabilizer are called mysterious. We establish that many other classes of acyclic quivers (including keys) often have mysterious points. This refutes \cite[Conjecture 1.1]{deep}, but establishes it in many important cases.
\end{abstract}

\maketitle

\section{Introduction}

An affine cluster variety $\var$ is the affine scheme $\var = \Spec(A)$ associated to a cluster algebra $A$ in the sense of Fomin--Zelevinsky \cite{FZ1}. Recall that the algebra $A$ has a distinguished set of generators which can be grouped into finite sets $\excluster{x}(\seed)$, where $\seed$ runs over the (typically infinite) set of \emph{seeds} of $A$. Each of the finite sets $\excluster{x}(\seed)$ has the same cardinality $r$ and, by the Laurent phenomenon \cite{LP}, induces an open embedding
\[
(\C^{\times})^{r} \hookrightarrow \var. 
\]
We call the image of this embedding a \emph{cluster torus}, denoted by $T(\seed)$. It is tempting to hope that $\var$ can be covered by the cluster tori. However, this is not always true. The \emph{deep locus} of $\var$ is  precisely the set of points that are not covered by the cluster tori. 

Note that $p \in \var$ is deep if and only if for every cluster $\excluster{x}(\seed)$ there exists some element $x \in \excluster{x}(\seed)$ such that $x(p) = 0$. If the set of seeds is infinite, verfying whether a point is deep can be a hard problem. A complete description of the deep points of a cluster variety exists in the following settings.

\begin{enumerate}
    \item Cluster algebras of simply-laced finite cluster type, \cite{deep} (see also \cite{beyer-muller} for type $A$)
    \item Cluster algebras associated to unpunctured surfaces, \cite{beyer-muller}.
    \item Rank $2$ cluster algebras \cite{beyer-muller}.
    \item The Markov cluster algebra \cite{beyer-muller, deep}.
    \item The cluster algebra associated to a family of positroid cells in $\mathrm{Gr}(3,n)$, including the top-dimensional one \cite{deep}. 
\end{enumerate}

All of these are obtained by \emph{ad-hoc} methods, special to the cluster algebras in question. 

There is an easy condition that guarantees a point is deep. The \emph{cluster dilation group} $\dil(\var)$ is a group that acts by dilations, hence freely, in every cluster torus.\footnote{This is usually called the \emph{cluster automorphism group} in the literature. See Remark \ref{rmk:terminology} for our reasons to not use this name.} So a sufficient condition for a point to be deep is that it has a nontrivial stabilizer under the action of $\dil(\var)$. This condition is not always necessary (see \cite[Proposition 3.22]{deep}). Nevertheless, in \cite[Conjecture 1.1]{deep} we conjecture that it is necessary for a large class of cluster algebras, namely, the locally acyclic ones. In this note we will establish this conjecture when $\var$ has a seed whose mutable part is a tree. However, we will also see that it fails to hold in general, even for cluster algebras that have an acyclic seed.

More precisely, following \cite{deep} let us 
say that a $p \in \var$ is \emph{mysterious} if $p$ is deep but $\dil(\var)$ acts freely on $p$. The first main result of this paper reads as follows. 

\begin{theorem}\label{thm:main-intro}
Assume $\var$ has a seed $\seed = (\exmatrix{B}(Q), \excluster{x})$ such that the mutable part of $Q$ is a tree. Then, $\var$ has no mysterious points.     
\end{theorem}

\begin{remark}
Following convention, trees do not have parallel arrows. So if there is more than one arrow between a pair of mutable vertices then $Q$ does \emph{not} satisfy the conditions in Theorem \ref{thm:main-intro}.
\end{remark}

Note that Theorem \ref{thm:main-intro} implies that cluster algebras of (simply-laced) finite cluster type have no mysterious points, thus recovering \cite[Theorem 5.19]{deep} using more elementary methods.

Let us explain our proof strategy. Assume first that $\var$ is any cluster variety of skew-symmetric type. Assume $p \in \var$ has a trivial stabilizer under the action of the cluster dilation group $\dil(\var)$. We need to find a cluster torus containing $p$. Fix a seed $\seed = (\exmatrix{B}(Q), \excluster{x})$, and assume that there exists a mutable vertex $i$ such that $x_i(p) \neq 0$. Consider the seed $\widetilde{\seed}$ obtained from $\seed$ by freezing the vertex $i$, and assume that the cluster variety $\widetilde{\var}$ coincides with the locus $\{x_i \neq 0\} \subseteq \var$. Then, we may consider $p \in \widetilde{\var}$ and work inductively on the number of mutable vertices. However, we have a natural embedding $\dil(\var) \subseteq \dil(\widetilde{\var})$, and it may happen that $\Stab_{\dil(\widetilde{\var})}(p)$ is nontrivial even if $\Stab_{\dil(\var)}(p)$ is (see Example \ref{ex:mysterious-point-intro} below for an instance of this). One sufficient condition that guarantees that $\Stab_{\dil(\widetilde{\var})}(p)$ is trivial is that $x_j(p) \neq 0$ for every vertex $j$ neighboring $i$ in $Q$. 

This suggests we should work with \emph{acyclic} seeds, and use the stratification of $\var$ studied by Lam and Speyer in \cite{LSII}. 
This stratification is indexed by the independent sets of vertices of the mutable part of~$Q$.
Using the strategy outlined above, we may reduce to the case where $p$ is in a stratum indexed by a maximal independent set.
Up to mutation equivalence, we may assume that every vertex is either a source or a sink, in particular, one-step mutations of $Q$ are still trees and we can perform a well-chosen sequence of mutations to take care of the case of maximally independent sets.

Outside of trees, the inductive strategy outlined above  may fail. Analyzing where it may fail we are led to a family of mysterious points.

\begin{theorem}\label{thm:main2-intro}
Consider the quiver
\[\begin{tikzcd}
	3 && 1 && 2
	\arrow["a", from=1-1, to=1-3] 
	\arrow["b"', from=1-5, to=1-3] 
\end{tikzcd}\]
Assume that $\gcd(a,b)=1$, $\min(a,b) \geq 2$. Then, the corresponding cluster algebra has mysterious points. 
\end{theorem}

\begin{remark}
    While the quiver $Q$ in Theorem \ref{thm:main2-intro} does not feature frozen variables, we will see below in Proposition \ref{prop:adding-frozen-variables} that Theorem \ref{thm:main2-intro} implies that we can add arbitrary frozen variables to $Q$ and the cluster algebra still has mysterious points. In particular, we have examples of acyclic cluster algebras with really full rank that have mysterious points. 
\end{remark}

\begin{example}\label{ex:mysterious-point-intro}
As an example, consider the quiver 
\[\begin{tikzcd}
	3 && 1 && 2
	\arrow["2", from=1-1, to=1-3]
	\arrow["3"', from=1-5, to=1-3]
\end{tikzcd}\]
We claim that any point $p$ in the locus $\{x_1 = x_1' = 0\}$ is mysterious. An example of such a point $p$ is taking $x_1(p) = x'_1(p) = 0$ and $x_2(p) = x_2'(p) = -1$ and $x_3(p) = x_3'(p) = 1$, so this locus is nonempty. Here, we are using $x_1, x_2, x_3$ for the initial seed, and $x_i'$ for the new cluster variable we obtain when mutating in direction $i$, see Section \ref{sec:acyclic}. Since $\gcd(2,3) = 1$, a straightforward computation shows that $p$ has trivial stabilizer under the action of the cluster dilation group, so to show it is mysterious we have to show it is deep. For this, note that upon freezing either of the variables $2$ or $3$, the point $p$ acquires a stabilizer. This means that if $\seed'$ is a seed such that $p \in T(\seed')$, $\seed'$ is at least one mutation at $2$ and one mutation at $3$ away from $\seed$. It is clearly also at least one mutation away from $\seed$. Now we use induction on the number of mutations away from $\seed$, and the fact that for any quiver $Q'$ mutation-equivalent to $Q$ one has $|b_{12}(Q')|, |b_{13}(Q')| \geq 2$ to show that $p$ does not belong to any cluster torus. See Proposition \ref{prop:so-may-deep} for details.  
\end{example}

Thus we have that \cite[Conjecture 1.1]{deep} is in general false. However, as shown by the results of this paper, it seems to be true in many cases. It'd be interesting to find conditions guaranteeing the non-existence of mysterious points. 

\begin{remark}
    The quiver in Theorem \ref{thm:main2-intro} is an example of a \emph{key}, a class of quivers studied in \cite{Ervin}. In fact, keys tend to have mysterious points. See Section \ref{subsec:keys} below.
\end{remark}

\begin{remark}
The cluster algebras considered in \cite{deep} (that do not have mysterious points) all come from positroid varieties in the Grassmannian \cite{GL-positroid} (which in turn come from Postnikov's plabic graphs \cite{postnikov}). It is shown in \cite{schwartz} that (up to frozen variables) every tree quiver also comes from a plabic graph. It is thus natural to conjecture that cluster algebras associated to plabic graphs do not have mysterious points. 
\end{remark}

The paper is organized as follows. In Section \ref{sec:cluster-varieties} we give preliminary results on cluster algebras/varieties and set up notations. In particular, we define the cluster dilation group and present it as a subgroup of a torus. In Section \ref{sec:trees} we prove Theorem \ref{thm:main-intro}. In Section \ref{sec:deep-rank-2} we discuss a method to decide whether a point is deep based on restriction to  rank $2$ subalgebras. This has the advantage of not referring to the stabilizer of a point and will ultimately lead us to produce a family of mysterious points in Section \ref{sec:mysterious-points}, in particular proving Theorem \ref{thm:main2-intro}. Moreover, in this section we exhibit several families of quivers with mysterious points, the simplest ones being those considered by Theorem \ref{thm:main2-intro}. 
These include quivers of arbitrarily large rank, as well as locally-acyclic quivers which are not mutation acyclic.

\subsection*{Acknowledgements} We thank Daniel Pérez Melesio and David Speyer for stimulating discussions. We also thank Marco Castronovo and Mikhail Gorsky for their comments on a preliminary version of the paper. The work of S.N. was partially supported by NSERC Discovery Grant RGPIN-2022-03960 and the Canada Research Chairs program, grant number CRC-2021-00120. The work of J.S. was partially supported by UNAM’s PAPIIT Grant IA102124 and SECIHTI Project C-F-2023-G-106.\\

This material is based upon work supported by the National Science Foundation under Grant No. DMS-1929284 while the authors were in residence at the Institute for Computational and Experimental Research in Mathematics in Providence, RI, during the Categorification and Computation in Algebraic Combinatorics semester program.

\section{Cluster varieties} \label{sec:cluster-varieties} 

We give preliminary results on cluster algebras and cluster varieties. 

\subsection{Cluster algebras} Let $n, m$ be non-negative integers, and let $\ambient$ be the field of rational functions in $n+m$ variables with complex coefficients.

Let $Q$ be an ice quiver with $n$ mutable vertices, labeled $\{1, \dots, n\}$, as well as $m$ frozen ones, labeled $\{n+1, \dots, n+m\}$. We will always assume that $Q$ has no loops or oriented $2$-cycles, and that there are no arrows between two frozen vertices.

We define an $(m+n)\times n$ matrix $\exmatrix{B}(Q) = (b_{ij})$ defined as follows:
\[
b_{ij} := \#\{\text{arrows $i \to j$}\} - \#\{\text{arrows $j \to i$}\}.
\]
We call $\exmatrix{B}(Q)$ the \emph{extended exchange matrix} of the quiver $Q$. Since $Q$ has no loops or oriented $2$-cycles and there are no arrows between frozen vertices, the data of $Q$ and of $\exmatrix{B}(Q)$ are equivalent. 
The number of mutable vertices $n$ is called the \emph{rank} of $Q$. We denote by $Q^{\mut}$ the quiver obtained after deleting all the frozen vertices and arrows incident to them, and call it the mutable part of $Q$. 

\begin{definition}
    A \emph{seed} in $\ambient$ is the data $\seed = (\exmatrix{B}, \excluster{x})$ where
    \begin{enumerate}
        \item $\exmatrix{B} = \exmatrix{B}(Q)$ is the extended exchange matrix of an ice quiver $Q$.
        \item $\excluster{x} = \{x_1, \dots, x_{m+n}\} \subseteq \ambient$ is such that $\ambient = \C(x_1, \dots, x_{m+n})$. 
    \end{enumerate}
    We call $\excluster{x}$ the \emph{extended cluster} of the seed $\seed$, and its elements are called \emph{cluster variables}. The cluster variables $x_1, \dots, x_n$ are called \emph{mutable} while the rest $x_{n+1}, \dots, x_{n+m}$ are called \emph{frozen}. 
\end{definition}

If $k$ is a mutable vertex of $Q$, the mutation $\mu_k(Q)$  of $Q$ at $k$ is defined to be the quiver obtained from~$Q$ by the following three-step procedure: (1) for each pair of arrows $i \to k \to j$ in $Q$ where $i$ and $j$ are not both frozen, add an arrow $i \to j$; (2) reverse all arrows incident to $k$; (3) delete a maximal collection of oriented $2$-cycles that may have been created in step (1).
Two quivers are \emph{mutation-equivalent} if they are related by a sequence of mutations.

\begin{definition}
    Let $\seed = (\exmatrix{B}, \excluster{x})$ be a seed, and let $1 \leq k \leq n$. The mutation $\mu_k(\seed)$ is the seed $\mu_k(\seed) = (\mu_k(\exmatrix{B}), \mu_k(\excluster{x}))$ where 
    \begin{enumerate}
        \item $\mu_k(\exmatrix{B})$ is the extended exchange matrix of the quiver $\mu_k(Q)$. 
        \item $\mu_k(\excluster{x}) = (x'_1, \dots, x'_{n+m})$, where $x'_i = x_i$ for $i \neq k$, and $x'_k = x_k^{-1}\left(\prod_{b_{ik} > 0}x_i^{b_{ik}} + \prod_{b_{ik} < 0}x_i^{-b_{ik}}\right)$. 
    \end{enumerate}
\end{definition}

We remark that mutation is involutive, that is, $\mu_k(\mu_k(\seed)) = \seed$. 

\begin{definition}
    Let $\seed = (\exmatrix{B}, \excluster{x})$ be a seed in $\ambient$. The \emph{cluster algebra} $A(\seed)$ is defined to be the $\C[x_{n+1}^{\pm1}, \dots, x_{n+m}^{\pm 1}]$-subalgeba of $\ambient$ generated by all the cluster variables of $\seed$ and all its iterated mutations. We denote by $\seeds(\seed)$ the set of all seeds obtained from $\seed$ by iterated mutations.
\end{definition}

The following result is essential in the theory of cluster algebras.

\begin{theorem}[The Laurent phenomenon \cite{LP}]
Let $\seed' \in \seeds(\seed)$. Then, any cluster variable of $\seed'$ can be written as a Laurent polynomial in the cluster variables of $\seed$. 
\end{theorem}

\subsection{Cluster varieties and the deep locus} 

\begin{definition}
    Let $\seed$ be a seed and $A(\seed)$ its associated cluster algebra. The affine variety
    \[
    \var(\seed) := \Spec(A(\seed))
    \]
    is called the \emph{cluster variety} associated to $\seed$. 
\end{definition}

Now let $\seed' = \big (\exmatrix{B'}, \excluster{z} = (z_1, \dots, z_n, z_{n+1}, \dots, z_{n+m})\big) \in \seeds(\seed)$ be a seed that can be obtained from $\seed$ by mutations. By the Laurent phenomenon, every cluster variable of $A(\seed)$ can be written as a Laurent polynomial in the cluster variables of $\seed'$, so that $A(\seed) \subseteq \C[z_1^{\pm 1}, \dots, z_{m+n}^{\pm 1}]$. Since, by definition, $\C[z_1, \dots, z_{m+n}] \subseteq A(\seed)$ we obtain
\begin{equation}\label{eq:cluster-tori}
A(\seed)[z_1^{-1}, \dots, z_{n+m}^{-1}] = \C[z_1^{\pm 1}, \dots, z_{n+m}^{\pm 1}]. 
\end{equation}
Thus, $\seed'$ determines an open subset $T(\seed')$ of $\var(\seed)$, defined by the non-vanishing of the variables $z_1, \dots, z_{m+n}$ and by \eqref{eq:cluster-tori} we have
\[
T(\seed') \cong (\C^{\times})^{n+m}.
\]

\begin{definition}
    The open tori $T(\seed')$ where $\seed' \in \seeds(\seed)$ are called the \emph{cluster tori} of the variety $\var(\seed)$. The \emph{cluster manifold} is
    \[
    \manifold(\seed) := \bigcup_{\seed' \in \seeds(\seed)}T(\seed') \subseteq \var(\seed). 
    \]
    The \emph{deep locus} of $\seed$ is
    \[
    \deep(\seed) := \var(\seed) \setminus \manifold(\seed). 
    \]
\end{definition}

In other words, an element $p \in \var(\seed)$ belongs to the deep locus if, for any seed $\seed' = (\exmatrix{B'}, \excluster{z}) \in \seeds(\seed)$, there exists $i \in \{1, \dots, n\}$ such that $z_i(p) = 0$. 

\subsection{The cluster dilation group} In \cite{deep} the authors observe a relationship between the deep locus of $\var(\seed)$ and the action of a certain group on $\var(\seed)$. 

\begin{definition}
    Let $\seed$ be a seed and $A(\seed)$ its corresponding cluster algebra. An algebra automorphism $\varphi: A(\seed) \to A(\seed)$ is said to be a \emph{cluster dilation} if, for every cluster variable (in every seed) $z$ of $A(\seed)$, there exists $\varphi_z \in \C^{\times}$ such that $\varphi(z) = \varphi_zz$. The set of all cluster dilations of $A(\seed)$ forms a group, called the \emph{cluster dilation group} $\dil(\seed)$. 
\end{definition}

\begin{remark}\label{rmk:terminology}
        The cluster dilation group $\dil(\seed)$ is called \emph{cluster automorphism group} in \cite{LSI, LSII, deep, GSV} among others. Since the term \emph{cluster automorphism group} is also used to name a different object (see \cite{cluster-auto, cluster-auto-2, cluster-iso-lawson, cluster-iso-cao} among others) we prefer the term \emph{cluster dilation group.} 
\end{remark}

By the Laurent phenomenon a cluster dilation is entirely determined by its values in a single extended cluster, so that in fact every seed $\seed' \in \seeds(\seed)$ determines an embedding $\dil(\seed) \hookrightarrow (\C^{\times})^{n+m}$. Following \cite{GSV}, we can determine the image of this embedding precisely, as follows. Let $\exmatrix{B'} = (b'_{ij})$ be the extended exchange matrix of the seed $\seed'$. 

\begin{lemma}[Lemma 2.3, \cite{GSV}]\label{lem:cluster-dilation-group} For any seed $\seed' = (\exmatrix{B'}, \excluster{z}) \in \seeds(\seed)$, the cluster dilation group $\dil(\seed)$ can be identified with the subgroup of $(\C^{\times})^{n+m}$ defined by the equations
\begin{equation}\label{eq:equations-dilation-group}
    \prod_{\substack{i = 1 \\ b'_{ik} > 0}}^{m+n}t_i^{b'_{ik}} = \prod_{\substack{i = 1 \\ b'_{ik} < 0}}^{m+n}t_i^{-b'_{ik}}, \qquad k = 1, \dots, n.
\end{equation}
\end{lemma}

Note that the equations \eqref{eq:equations-dilation-group} are obtained simply by requiring that the mutated variable \[x_k^{-1}\left(\prod_{b'_{ik}>0} x_i^{b'_{i_k}} + \prod_{b'_{ik}<0}x_i^{-b'_{ik}}\right)\] is homogeneous. 

\begin{example}\label{ex:dilation-group}
Consider the quiver (without frozen variables)
\[\begin{tikzcd}
	3 && 1 && 2
	\arrow["2", from=1-1, to=1-3]
	\arrow["3"', from=1-5, to=1-3]
\end{tikzcd}\]
so that we have the following equations
    \[\begin{array}{ccc}
    t_1^2 = 1, & t_2^3t_3^2 = 1, & t_1^3 = 1,      
    \end{array}
    \]
corresponding to the vertices $3, 1$ and $2$, respectively. Note that the first and third equations imply that $t_1 = 1$, so the cluster dilation group is simply $\{(t_2, t_3) \in (\C^{\times})^2 \mid t_2^3t_3^2 = 1\}$. 
\end{example}

By definition, $\dil(\seed)$ acts on $A(\seed)$  by algebra automorphisms so it also acts on $\var(\seed)$. Also by definition, $\dil(\seed)$ acts by dilations, hence freely, on any cluster torus $T(\seed') \subseteq \var(\seed)$. In other words, any element $p \in \var(\seed)$ with non-trivial stabilizer under the action of $\dil(\seed)$ must be deep. 

\begin{definition}
    Let $\seed$ be a seed. We define the stabilizer locus
    \[
    \stab(\seed) := \{p \in \var(\seed) \mid \Stab_{\dil(\Sigma)}(p) \neq \{1\}\}. 
    \]
\end{definition}

    From the discussion above,
    \[
    \stab(\seed) \subseteq \deep(\seed).
    \]
This inclusion need not be an equality, and we call a point $p \in \var(\seed)$ \emph{mysterious} if $p \in \deep(\seed) \setminus \stab(\seed)$. Thus, $p$ is mysterious if it is deep, but its deepness cannot be explained by looking at the cluster dilation group. 



\subsection{Acyclic case}\label{sec:acyclic}
Let $\seed = (\exmatrix{B}, \excluster{x})$ be a seed with associated quiver $Q$ and assume that the mutable part $Q^{\mut}$ is \emph{acyclic}, that is, it does not have directed cycles. 
In this case, the cluster algebra $A(\seed)$ admits an explicit description by generators and relations. Namely, by \cite[Corollary 1.21]{BFZ} we have that 
\begin{equation}\label{eq:acyclic-cluster-algebra}
A(\seed) \cong \C[x_1, x'_1, \dots, x_n, x_n', x_{n+1}^{\pm1}, \dots, x_{n+m}^{\pm1}]\Big/\left(x_ix_i' - \prod_{j = 1}^{n+m}x_{j}^{[b_{ji}]_{+}} - \prod_{j = 1}^{n+m}x_j^{[-b_{ji}]_{+}}, i = 1, \dots, n\right)
\end{equation}
and consequently we have the following description for the cluster variety
\begin{equation}\label{eq:acyclic-cluster-variety}
\var(\seed)\! \cong \!\left\{\!(p_1, p'_1, \dots, p_n, p'_n, p_{n+1}, \dots, p_{n+m})\! \in\! \C^{2n} \!\times \! (\C^{\times})^{m} \middle| p_ip'_i \!=\! \prod_{j = 1}^{n+m}p_{j}^{[b_{ji}]_{+}} \!\!+\!\! \prod_{j = 1}^{n+m}p_j^{[-b_{ji}]_{+}}\!, i = 1, \dots, n\! \right\}.
\end{equation}

    In \eqref{eq:acyclic-cluster-algebra}, the element $x'_i$ is the unique element in $\mu_i(\excluster{x}) \setminus \excluster{x}$, and the relations are the exchange relations defining $x'_i$. 

    In \cite[Section 3]{LSII}, Lam and Speyer obtain a stratification of $\var(\seed)$ whose strata are particularly nice when $\exmatrix{B}$ has \emph{really full rank}, i.e., when the rows of $\exmatrix{B}$ span $\Z^n$ as a $\Z$-module. Recall that a subset $I$ of the vertices of $Q$ is an \emph{independent set} if there are no arrows between any two of its members. 

    \begin{theorem}\label{thm:def-mathcalO}
Let $\seed$ be a seed such that $Q^{\mut}$ is acyclic. For each independent set $I \subseteq \{1,\dots,n\}$, define
\[
\mathcal{O}_I := \{(p_1, p'_1, \dots, p_n, p'_n, p_{n+1}, \dots, p_{n+m}) \in \var(\seed) \mid p_i = 0 \forall i \in I, \, p_j \neq 0 \forall j \notin I\}. 
\]
Then
\[
\var(\seed) = \bigsqcup_{\substack{I \subseteq \{1,\ldots n\} \\ \text{is independent}}} \mathcal{O}_I. 
\]
If, moreover, $\exmatrix{B}$ has really full rank then
\[
\mathcal{O}_I \cong \C^{|I|} \times (\C^{\times})^{n+m-2|I|}
\]
The coordinates on the $\C^{|I|}$ factor are $p'_i$ for $i \in I$. 
    \end{theorem}

    Note that when $I = \emptyset$ the stratum $\mathcal{O}_I$ is precisely the cluster torus $T(\seed)$ corresponding to the acyclic seed $\seed$. 

We can use the explicit description \eqref{eq:acyclic-cluster-variety} to prove the following result.

\begin{lemma}\label{lem:creating-deep-points-frozens}
    Let $\seed$ be a seed such that $Q^{\mut}$ is acyclic, and let $\widetilde{\seed}$ be obtained from $\seed$ by adding $\ell$ arbitrary frozen rows to the exchange matrix. Let 
    \[
    p = (p_1,p'_1,\dots, p_n, p'_n, p_{n+1}, \dots, p_{n+m}) \in \var(\seed)
    \]
    and define the point
    \[
    \widetilde{p} = (p_1, p'_1,\dots, p_n, p'_n, p_{n+1}, \dots, p_{n+m}, 1, \dots, 1) \in \C^{2n} \times (\C^{\times})^{m+\ell}.
    \]
    Then
    \begin{enumerate}
        \item $\widetilde{p} \in \var(\widetilde{\seed})$.
        \item $p$ belongs to a cluster torus in $\var(\seed)$ if and only if $\widetilde{p}$ belongs to a cluster torus in $\var(\widetilde{\seed})$.
        \item The groups $\Stab_{\dil(\seed)}(p)$ and $\Stab_{\dil(\widetilde{\seed})}(\widetilde{p})$ are isomorphic.
    \end{enumerate}
\end{lemma}
\begin{proof}
An exchange relation in $\seed$ is of the form $x_kx_k' = M_1 + M_2$ and an exchange relation in $\widetilde{\seed}$ is of the form $x_kx_k' = M'_1 + M'_2$, where $M_i$ differs from $M'_i$ by a monomial in the extra frozen variables in $\widetilde{\seed}$. From this observation, (1) is clear. Statement (2) follows from \cite[Lemma 3.3.2]{FWZ}. Finally, let $\widetilde{t} = (t_1, \dots, t_{n+m}, t_{n+m+1}, \dots, t_{n+m+\ell}) \in \Stab_{\dil(\widetilde{\seed})}(\widetilde{p})$. Since the last $\ell$ coordinates of $\widetilde{p}$ are nonzero, we have $t_{n+m+1}, \dots, t_{n+m+\ell} = 1$. It is then easy to see that $(t_1, \dots, t_{n+m}) \mapsto (t_1, \dots, t_{n+m}, 1, \dots, 1)$ defines an isomorphism $\Stab_{\dil(\seed)}(p) \to \Stab_{\dil(\widetilde{\seed})}(\widetilde{p})$. 
\end{proof}

\section{Tree cluster algebras}\label{sec:trees}

\subsection{Statement of the theorem} The goal of this section is to prove the following result. 

\begin{theorem}\label{thm:main-notintro}
Let $\seed = (\exmatrix{B}(Q), \excluster{x})$ be a seed such that $Q^{\mut}$ is a tree. Then, $\var(\seed)$ has no mysterious points. 
\end{theorem}

We will prove Theorem \ref{thm:main-notintro} by induction on $n$, the number of vertices of the quiver $Q^{\mut}$. However, first we make a few observations that will allow us to assume that every vertex of $Q$ is a sink or a source, so that the exchange relations are simplified.

\subsection{Reductions}\label{sec:reductions} The next result tells us, essentially, that we are free to choose the frozen variables in $Q$ as long as the quiver $Q$ has really full rank. 

\begin{lemma}\label{lem:can-choose-frozens}[Corollary 3.19, \cite{deep}]
Assume Theorem \ref{thm:main-notintro} is true for some choice of frozen variables that makes $\exmatrix{B}(Q)$ have really full rank. Then, Theorem \ref{thm:main-notintro} is true for an arbitrary choice of frozen variables.
\end{lemma}

It is well-known, and easy to show, that any two orientations of a tree can be related by reversing arrows at sinks and sources. Note that mutation at a sink (resp. source) simply amounts to reversing the arrows incident to that sink (resp. source). Thus, to show Theorem \ref{thm:main-notintro} it is enough to assume that $Q$ satisfies the following:
\begin{itemize}
    \item $Q^{\mut}$ is a tree where every vertex is either a source or a sink; 
    \item there is exactly one frozen vertex $\overline{i}$ for each mutable vertex $i$ of $Q^{\mut}$; 
    \item there is just one arrow from each frozen vertex to a mutable vertex: we have an arrow $\boxed{\overline{i}} \to i$ if $i$ is a sink in $Q^{\mut}$, and an arrow $i \to \boxed{\overline{i}}$ if $i$ is a source in $Q^{\mut}$. (I.e., a mutable vertex is a sink (resp. source) of $Q$ if and only if it is a sink (resp. source) in $Q^{\mut}$.)
\end{itemize}

See Figure \ref{fig:quiver} for an example of such a quiver. 
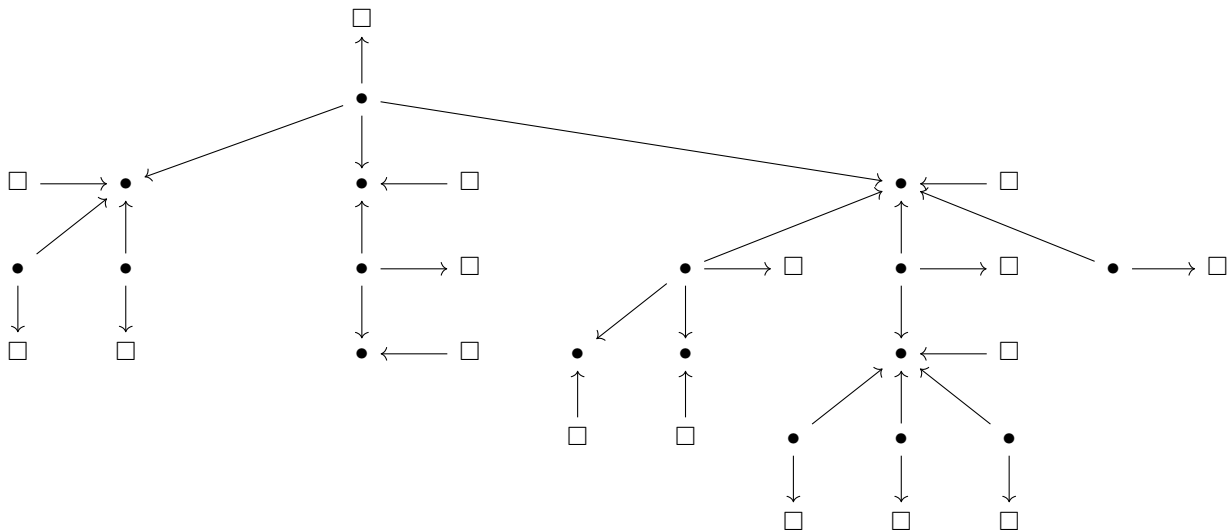
\begin{figure}
    \centering
    
    \[\begin{tikzcd}
	&&&& \square &&&&&&&& \\
	&&&& {\bullet} \\
	\square & \bullet &&& \bullet & \square &&&& \bullet & \square \\
	\bullet & \bullet &&& \bullet & \square && \bullet & \square & \bullet & \square & \bullet & \square \\
	\square & \square &&& \bullet & \square & \bullet & \bullet && \bullet & \square \\
	&&&&&& \square & \square & \bullet & \bullet & \bullet \\
	&&&&&&&& \square & \square & \square
	\arrow[from=2-5, to=1-5]
	\arrow[from=2-5, to=3-2]
	\arrow[from=2-5, to=3-5]
	\arrow[from=2-5, to=3-10]
	\arrow[from=3-1, to=3-2]
	\arrow[from=3-6, to=3-5]
	\arrow[from=3-11, to=3-10]
	\arrow[from=4-1, to=3-2]
	\arrow[from=4-1, to=5-1]
	\arrow[from=4-2, to=3-2]
	\arrow[from=4-2, to=5-2]
	\arrow[from=4-5, to=3-5]
	\arrow[from=4-5, to=4-6]
	\arrow[from=4-5, to=5-5]
	\arrow[from=4-8, to=3-10]
	\arrow[from=4-8, to=4-9]
	\arrow[from=4-8, to=5-7]
	\arrow[from=4-8, to=5-8]
	\arrow[from=4-10, to=3-10]
	\arrow[from=4-10, to=4-11]
	\arrow[from=4-10, to=5-10]
	\arrow[from=4-12, to=3-10]
	\arrow[from=4-12, to=4-13]
	\arrow[from=5-6, to=5-5]
	\arrow[from=5-11, to=5-10]
	\arrow[from=6-7, to=5-7]
	\arrow[from=6-8, to=5-8]
	\arrow[from=6-9, to=5-10]
	\arrow[from=6-9, to=7-9]
	\arrow[from=6-10, to=5-10]
	\arrow[from=6-10, to=7-10]
	\arrow[from=6-11, to=5-10]
	\arrow[from=6-11, to=7-11]
\end{tikzcd}\]

    \caption{An example of the sort of quivers we consider to prove Theorem~\ref{thm:main-notintro}. Note that every vertex is either a sink or a source. Frozen vertices are marked as boxes. }
    \label{fig:quiver}
\end{figure}

\subsection{Proof of Theorem \ref{thm:main-notintro}}
Now we proceed to prove Theorem \ref{thm:main-notintro}. We work by induction on the number $n$ of mutable vertices of $Q$, the base case $n = 1$ is clear. We make the reductions above and assume $Q$ has the form specified in Section \ref{sec:reductions}. In particular, $Q$ has $n$ mutable vertices (labeled $1, \dots, n$) and $n$ frozen vertices as well (labeled $\overline{1}, \dots, \overline{n}$), and every vertex is a source or a sink.

We present the algebra $A(\seed)$ as in \eqref{eq:acyclic-cluster-algebra}. In particular, for $k = 1, \dots, n$ we let $x'_k$ be the new variable obtained after mutating $\seed$ in direction $k$, so that $A(\seed)$ is generated by $x_1,x'_1, \dots, x_n, x'_n, x_{\overline{1}}^{\pm1}, \dots, x_{\overline{n}}^{\pm 1}$.  The cluster dilation group is identified with the set of elements $t = (t_1, \dots, t_n, t_{\overline{1}}, \dots, t_{\overline{n}}) \in (\C^{\times})^{2n}$ satisfying:
\begin{equation}\label{eq:cluster-dilation-case-of-interest}
\prod_{j \; \text{a neighbor of $k$}} t_j = 1 \qquad 
 \text{for every mutable vertex} \; k.
 \end{equation}

Note that $t(x_{k}) = t_kx_k, t(x_k')= t_k^{-1}x_k'$ 
and $t(x_{\overline{k}}) = t_{\overline{k}}x_{\overline{k}}$ for $k = 1, \dots, n$. 
 
Let $p \in \var(\seed)$, and assume that $\Stab_{\dil(\seed)}(p)$ is trivial. We must find a seed $\seed' \in \seeds(\seed)$ such that $p \in T(\seed')$. We have that $p \in \mathcal{O}_I$ for some independent set $I$ of the mutable vertices of $Q$. We separate in two cases. \\

{\bf Case 1. $I$ is not a maximal independent set.} In this case, there exists a mutable vertex $k$ of $Q$ such that $x_k(p) \neq 0$, and $x_j(p) \neq 0$ for every neighbor $j$ of $k$---indeed, any vertex $k\notin I$ such that $I\cup\{k\}$ is independent works. By \cite[Lemma 3.4]{muller}, the localization $A(\seed)[x_k^{-1}]$ is the cluster algebra of the seed $\widetilde{\seed}$ obtained from $\seed$ by freezing the $k$-th variable. So we may think of $p$ as being an element of $\var(\widetilde{\seed})$. Note that we have $\dil(\seed) \subseteq \dil(\widetilde{\seed}) \subseteq (\C^{\times})^{2n}$, and $\dil(\seed)$ is obtained from $\dil(\widetilde{\seed})$ by imposing the relation \eqref{eq:cluster-dilation-case-of-interest} at the vertex $k$. However, since $x_j(p) \neq 0$ for every neighbor $j$ of $k$, we have that $t_j = 1$ if $j$ is a neighbor of $k$ and $t \in \Stab_{\dil(\widetilde{\seed})}(p)$, so that relation \eqref{eq:cluster-dilation-case-of-interest} at vertex $k$ is already valid in $\Stab_{\dil(\widetilde{\seed})}(p)$. Thus, this stabilizer is trivial, and we may use induction to conclude that $p$ belongs to a cluster torus of $\var(\widetilde{\seed})$, that is also a cluster torus of $\var(\seed)$.  \\

{\bf Case 2. $I$ is a maximal independent set.} 
We may assume that $x'_i(p) = 0$ for every $i \in I$, as otherwise we may apply $\mu_i$ and freeze vertex $i$ ($\mu_i(Q)$ is a tree since $i$ is either a sink or a source; cf. Case~1).
We separate into several cases.  \\

{\bf Case 2.1. There exists a leaf of $Q^{\mut}$ that does not belong to $I$}. Let $j$ be this leaf. Note that its unique neighbor $i$ in $Q^{\mut}$ must belong to $I$ as $I$ is maximal. 
We perform one mutation at $j$ followed by one mutation at $i$. Mutating at $j$ we have
\[
x'_{j}(p) = \frac{1 + x_{\overline{j}}(p)x_{i}(p)}{x_{j}(p)} = \frac{1}{x_{j}(p)} \neq 0. 
\]
Note that $i$ stops being a sink or a source. Performing a mutation at $i$ we obtain a new cluster variable
\[
 x_{i}'' = \frac{x'_{j} + M}{x_{i}} = \frac{\frac{1 + x_{\overline{j}}x_{i}}{x_{j}} + M}{x_{i}} = \frac{1 + x_{\overline{j}}x_{i} + x_{j}M}{x_{j}x_{i}} = \frac{x_{i}' + x_{\overline{j}}}{x_{j}} 
\]
where $M$ is the product of all $x_{k}$ where $k$ runs over the neighbors different from $j$ of $i$ in $Q$. We see that $x_{i}''(p) = \frac{x_{\overline{j}}(p)}{x_{j}(p)} \neq 0$. 
The neighbors of $i$ in $\mu_{i}(\mu_{j}(Q))$ are the same as its neighbors in $Q$, and we have just shown that the corresponding cluster variables in the seed $\mu_{i}(\mu_{j}(\seed))$ do not vanish. Moreover, after freezing $i$ we obtain a quiver whose mutable part is a tree with $n-1$ mutable vertices. Thus we may apply induction as in Case 1. \\

{\bf Case 2.2. All leaves of $Q^{\mut}$ belong to $I$.} We again separate in two subcases. \\

{\bf Case 2.2.1. The set $I$ consists of all sink (resp. source) vertices of $Q^{\mut}$. } 
Assume without loss of generality that $I$ consists of all the sinks. Thus, $x_i(p) = 0$ for every mutable sink $i$ and $x_j(p) \neq 0$ for every mutable source $p$. We claim that:

\begin{framed}
\begin{center}
There is a nontrivial element in $\Stab_{\dil(\seed)}(p)$.
\end{center}
\end{framed}

Note this claim contradicts our first assumption regarding $p$.
Let $t \in \Stab_{\dil(\seed)}(p)$. Note that $t_{\overline{i}} = 1$ for every frozen vertex $\overline{i}$, since $x_{\overline{i}}(p) \neq 0$. For the same reason, $t_{i} = 1$ for every source vertex $i$. 
Choose a source vertex $j$ (observe that $j \not \in I$). 
Now let $i_1, \dots, i_s$ be the mutable neighbors of $j$. Note that these all belong to $I$. By assumption $s > 1$ (otherwise $j \not \in I$ and a leaf). Consider any solution to the equation $t_{i_1}\cdots t_{i_s} = 1$ with not all $t_{i_k}$ being $1$. Note that we can propagate this solution in order to find values of $t_i$ for every sink $i$ satisfying \eqref{eq:cluster-dilation-case-of-interest} at every mutable source $k$: since there are no source leaves, any such equation will involve two sinks which are at different distances from~$j$, and this guarantees that the solution can be propagated. Since not all $t_{i_k}$ are equal to $1$ and $x'_{i}(p) = 0$ for every $i \in I$, this gives a nontrivial element in $\Stab_{\dil(\seed)}(p)$, arriving at the desired contradiction.

{\bf Case 2.2.2. There exists both a sink and a source vertex in $I$.} 
Because $I$ is maximal and the mutable part of our quiver is a connected tree, there must be some adjacent pair of vertices $j_0, j_1$ which are both not in $I$. 
Thus the set $S := \{(j_0,j_1) \mid j_i \not \in I, j_0 \leftarrow j_1\}$ is nonempty.
Fix a root vertex $r$ of $Q^{\mut}$, and choose a pair $(j_0,j_1)$ whose distance to the root $r$ is maximal.
Without loss of generality, say $j_0$ is further from $r$ than $j_1$.
Because $I$ is maximal, there is a vertex $i_0 \in I$ adjacent to $j_0$.

In fact, the only neighbors of $j_0$ are $i_0$ and $j_1$.
Suppose for contradiction that $i_1$ is also adjacent to~$j_0$.
If $i_1 \not \in I$ then  $(j_0,i_1) \in S$ and has larger distance from $r$ than $(j_0, j_1)$ (contradicting our choice of~${(j_0, j_1)}$).
If $i_1 \in I$ then $p$ has a nontrivial stabilizer: 
for $j \in Q^{\mut}$, set $t_j = 1$ if $j$ does not belong to the trees of descendants (away from $r$) of $i_0$ or $i_1$. 
Set $t_{i_0} = t_{i_1} = -1$, and propagate this down the trees of descendants of $i_0$ and $i_1$, similarly to Case 2.2.1, in order to find a nontrivial element in the  stabilizer of~$p$, a contradiction.  
Note that we can propagate precisely because of the maximality of the distance from $j_0$ to $r$: down the trees of descendants of $i_0$ or $i_1$, $I$ consists of all the sources.

Let $k_0, \dots, k_{t}$ the children of $i_0$. We have the following configuration:
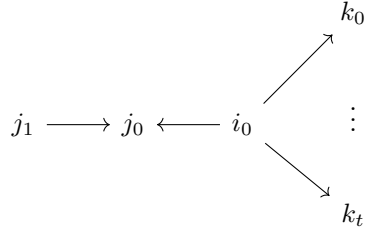
\begin{center}
\begin{tikzcd}
 & & & k_0 \\ 
 j_1 \ar[r] & j_0 & i_0 \ar[l] \ar[dr] \ar[ur] & \vdots \\ 
 & & & k_t   
\end{tikzcd}
\end{center}
Now we perform $\mu_{j_0}$ followed by $\mu_{i_0}$. We see that $x'_{j_0}(p) \neq 0$, and the quiver $\mu_{j_0}(Q)^{\mut}$ near $j_0$ is:
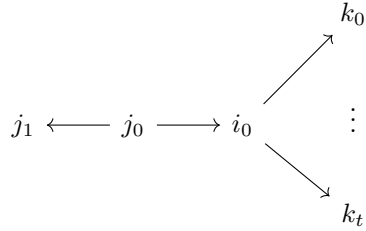
\begin{center}
\begin{tikzcd}
 & & & k_0 \\ 
 j_1  & j_0 \ar[r] \ar[l]& i_0 \ar[dr] \ar[ur] & \vdots \\ 
 & & & k_t   
\end{tikzcd}
\end{center}
Next, mutating at $i_0$ we have:
\[
x''_{i_0} = \frac{x'_{j_0} + M}{x_{i_0}} = \frac{\frac{1 + x_{\overline{j_0}}x_{j_1}x_{i_0}}{x_{j_0}} + M}{x_{i_0}} = \frac{1 + x_{\overline{j_0}}x_{j_1}x_{i_0} + x_{j_0}M}{x_{i_0}x_{j_0}} = \frac{x_{i_0}' + x_{\overline{j_0}}x_{j_1}}{x_{j_0}}
\]
where $M = x_{\overline{i_0}}x_{k_0}\cdots x_{k_t}$. We obtain that $x''_{i_0}(p) = \frac{x_{\overline{j_0}}(p)x_{j_1}(p)}{x_{j_0}(p)} \neq 0$, since we are assuming that $j_1 \notin I$. The quiver $\mu_{i_0}(\mu_{j_0}(Q))^{\mut}$ locally around $j_0$ is:
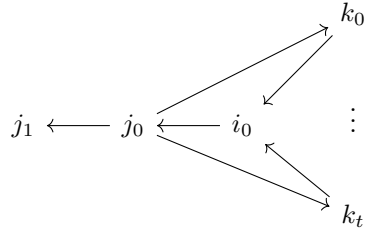
\begin{center}
\begin{tikzcd}
 & & & k_0 \ar[dl] \\ 
 j_1  & j_0 \ar[l] \ar[urr] \ar[drr]& \ar[l] i_0 & \vdots \\ 
 & & & k_t \ar[ul]
\end{tikzcd}
\end{center}
The variables associated to the mutable neighbors of $i_0$ in the seed $\mu_{i_0}(\mu_{j_0}(\seed))$ are $x'_{j_0}$, $x_{k_0}, \dots, x_{k_t}$. As we have seen, $x'_{j_0}(p) \neq 0$, and by the independence of $I$ it is clear that $x_{k_0}(p), \dots, x_{k_t}(p) \neq 0$. Finally, the variable associated to $i_0$ is $x''_{i_0}$ and we have seen that $x''_{i_0}(p) \neq 0$, so we may freeze the vertex $i_0$ without changing the stabilizer. After freezing $i_0$, we obtain a quiver whose mutable part is a tree. So we may apply induction as in the Case 1. This finishes the proof of Theorem \ref{thm:main-notintro}. 
\hfill $\square$

\begin{corollary}
    Let $\seed = (\exmatrix{B}(Q), \excluster{x})$ be a seed such that $Q^{\mut}$ is a tree. Assume that the exchange matrix $\exmatrix{B}(Q^{\mut})$ has really full rank. Then $\deep(\seed) = \emptyset$.  
\end{corollary}
\begin{proof}
Let $\widetilde{\seed}$ be the seed corresponding to deleting all frozen variables from $\seed$. By our assumption,~$\dil(\widetilde{\seed})$ is trivial, and by Theorem \ref{thm:main-notintro} this implies that $\deep(\widetilde{\seed}) = \emptyset$. The result follows from \cite[Corollary 3.7]{deep}. 
\end{proof}

\section{Deep points from rank 2 subalgebras}\label{sec:deep-rank-2}

The goal of this section is proving Proposition \ref{prop:so-may-deep} that, roughly speaking, produces deep points by looking at rank $2$ subalgebras of a given cluster algebra. We start with a study of deep points of rank~$2$ cluster algebras. The following result can be recovered from results in \cite{beyer-muller}, but we present a proof for completeness and to illustrate the techniques in this paper.


\begin{lemma}\label{lem:rank2}
Rank $2$ cluster algebras have no mysterious points.
\end{lemma}
\begin{proof}
By \cite[Corollary 3.19]{deep} it is enough to show the proposition assuming we have a quiver of the following form
\[\begin{tikzcd}
	{\boxed{\overline{1}}} && 1 && 2 && {\boxed{\overline{2}}}
	\arrow[from=1-3, to=1-1]
	\arrow["a", from=1-3, to=1-5]
	\arrow[from=1-7, to=1-5]
\end{tikzcd}\]
We will assume $a > 1$, the case $a = 1$ can be verified directly. 
The cluster dilation group is given by $\dil(\seed) = \{t = (t_1, t_2, t_{\overline{1}}, t_{\overline{2}}) \in (\C^{\times})^{4} \mid t_1^{a}t_{\overline{2}} = t_2^{a}t_{\overline{1}} = 1\}$. Let $p \in \var(\seed)$ have a trivial stabilizer under the action of $\dil(\seed)$. If $p \in \mathcal{O}_{\emptyset}$ then $p$ is already in the cluster torus defined by the seed $\seed$. If $p \in \mathcal{O}_{\{2\}}$ but $x_2'(p) \neq 0$, then $p$ is in the cluster torus defined by $\mu_2(\seed)$. If $p \in \mathcal{O}_{\{2\}}$ and $x_2'(p) = 0$, then $(1, \zeta, 1, 1) \in \Stab_{\dil(\seed)}(p)$ is a nontrivial element, where $\zeta$ is a primitive $a$-th root of unity, so this case cannot happen. The case where $p \in \mathcal{O}_{\{1\}}$ is analogous. In any case, $p$ belongs to a cluster torus.
\end{proof}

\begin{remark}
        Note that, in Lemma \ref{lem:rank2}, the cluster algebra $A(\seed)$ is of infinite cluster type if $a > 1$. However, the proof shows that the three cluster tori $T(\seed), T(\mu_1(\seed))$ and $T(\mu_2(\seed))$ are enough to cover the cluster manifold $\manifold(\seed)$. This phenomenon is \emph{not} true for $a = 1$, where we need all five tori to cover the cluster manifold. Let us briefly show this. Since the quiver $1 \to 2$ is already really full rank, for simplicity we will work without frozen variables, so that $\var(\seed) = \{(p_1, p_1', p_2, p_2') \mid p_1p_1'= 1 + p_2, p_2p_2' = 1 + p_1\}$. There are five seeds in $\seeds(\seed)$, whose clusters are given as follows:
        \[
        \excluster{x} = \{x_1, x_2\},\; \mu_1(\excluster{x}) = \{x_1', x_2\}, \;\mu_2(\excluster{x}) = \{x_1, x_2'\}, \; \mu_2\mu_1(\excluster{x}) = \{x_1', x\}, \; \mu_2\mu_1(\excluster{x}) = \{x, x_2'\}
        \]
        where $x = \frac{1 + x_1 + x_2}{x_1x_2} = x'_1x'_2 - 1$. For each of the five cluster tori, we exhibit a point that belongs to this cluster torus but does not belong to any of the other four cluster tori.
        \[
        \begin{array}{|c|c|c|}
        \hline 
        \text{Seed} & \text{Cluster} &  (p_1, p_1', p_2, p_2') \\
        \hline
        \seed & \{x_1, x_2\} & (-1, 0, -1, 0) \\
        \hline
        \mu_1(\seed) & \{x_1', x_2\} &  (0, -1, -1, -1) \\
        \hline
        \mu_2(\seed) & \{x_1, x'_2\} &  (-1, -1, 0, -1) \\
        \hline
        \mu_2\mu_1(\seed) & \{x_1', x_1'x_2' - 1\} & (-1, -1, 0, 0) \\
        \hline
        \mu_2\mu_1(\seed) & \{x'_1x'_2 - 1, x'_2\} & (0, 0, -1, -1) \\
        \hline
        \end{array}
        \]
        So we see that, in spite of \eqref{eq:acyclic-cluster-algebra} and \eqref{eq:acyclic-cluster-variety}, if the mutable part of $\seed$ is acyclic 
        it is in general not enough to take the tori of $\seed$ and all its neighbors to cover $\manifold(\seed)$. Moreover see \cite{dani}, where it is shown that at least $\frac{2^{n+2}+(-1)^{n-1}}{3} - \mathbf{1}_{\mathrm{odd}}(n)$ cluster tori are needed to cover a cluster manifold of type $A_n$, where $\mathbf{1}_{\mathrm{odd}}$ is the indicator function on the set of odd numbers. 
\end{remark}

\begin{lemma}[{\cite[Corollary 4.2]{beyer-muller}}]\label{lem:zeroes-sticky-rank-2}
Suppose $p$ is deep in a rank $2$ cluster algebra $\var$ with connected mutable part $1 \stackrel{a}{\rightarrow} 2$, and that $x_{1}(p)=0$ in some $\seed$. 
Then $x_{1}(p)=0$ and $x_{2}(p) \neq 0$ in every seed $\seed' \in \seeds(\seed)$. 
(Indeed, up to relabeling vertices $1$ and $2$, all deep points are of this form.)
\end{lemma}
\begin{proof}
By \cite[Corollary 5.4]{muller}, $x_{1}(p)$ and $x_{2}(p)$ cannot simultaneously vanish (the edge between them is a separating edge). 
The coefficient variables $x_{\bar{i}}(p)\neq0$, so the exchange relation for $2$ is $x'_{2}(p) = \frac{M + 0}{x_{2}(p)} \neq 0.$
As $p$ is deep, something must be $0$ after each mutation (otherwise $p$ is in a torus). The parenthetical statement follows immediately from the definition of deep points and the previous line. 
\end{proof}




We are now ready to show our main tool for constructing deep points.

\begin{proposition}\label{prop:so-may-deep}
For a seed $\seed  = (\exmatrix{B}, \exmatrix{x})$, suppose that for all $\seed' = (\exmatrix{B}', \exmatrix{z}) \in \seeds(\seed)$, we have $|b_{1,i}'| \geq 2$ for $i$ in $2, \dots, n$. 
Fix $p\in \var(\seed)$ such that $x_1(p) = x'_1(p)=0$ and $x_j(p) \neq 0$ for $j\neq1$. 
Then $p$ is deep. 
\end{proposition}


\begin{proof}
If the rank $n=1$, then this is a brief computation. 
If $n=2$, 
then $(\zeta, 1, \ldots)$, where $\zeta$ is a $|b_{12}|$-root of unity, belongs to the stabilizer of $p$.
So we consider the case $n \geq 3$. 

Let $x_{\seed'i}$ denote the cluster variable associated to vertex $i$ in the seed $\seed'\in\seeds(\seed)$. 
We show for every seed that $x_{\seed'1}(p)=0$ and $x_{\seed'j}(p) \neq 0$ for $j \neq 1$ (thus $p \not \in T(\seed')$). 
We argue by induction on the number $d$ of mutations from our initial seed $\seed$. 

By construction, $p\not \in T(\seed)$ nor $T(\mu_1(\seed))$.
Further, $p \not \in T(\mu_i(\seed))$ as the new cluster variable satisfies: \[
x'_{i}(p) = \frac{(M_1 + M_2)(p)}{x_{i}(p)}
\]
where $M_1 + M_2$ is a binomial in the $x_{j}$'s for $j \neq i$. 
By assumption, $|b'_{1i}|\geq 2$ so we know that $x_{1}$ divides $M_1$ or $M_2$ (but not both), while the other monomial is either $1$ or a product of powers of the $x_{j}$'s. In any case, we get $x_{\seed'i}(p) \neq 0$.
So the claim holds for $d\leq1$.

Suppose we have the claim for distance $d$, let us show it for $d+1$. 
We have the following path between the seeds $\seed = \seed_0$ and $\seed' = \seed_{d+1}$:
\[\begin{tikzcd}
	{\seed_0} & {\seed_1} & \cdots & {\seed_{d-1}} & {\seed_{d}} & {\seed_{d+1}}
	\arrow["{i_1}", no head, from=1-1, to=1-2]
	\arrow["{i_2}", no head, from=1-2, to=1-3]
	\arrow["{i_{d-1}}", no head, from=1-3, to=1-4]
	\arrow["{i_{d}}", no head, from=1-4, to=1-5]
	\arrow["{i_{d+1}}", no head, from=1-5, to=1-6]
\end{tikzcd}\]
Assume $i = i_{d+1} \neq 1$. Then, we have to show that $x_{\seed' i}(p) \neq 0$. But this follows from the same argument given for $p \not \in T(\mu_i(\seed))$ above (replacing $x_i$ and $x_i'$ with $x_{\seed_{d}i}$ and $x_{\seed'i}$ respectively). 

It remains to treat the case where $i=i_{d+1} = 1$. We have to show that $x_{\seed' 1}(p) = 0$. We may assume that $i_{d} \neq 1$. We consider a different seed $\seed''$:
\[\begin{tikzcd}
	&&&& {\seed''} & \\
	{\seed_0} & {\seed_1} & \cdots & {\seed_{d-1}} & {\seed_{d}} & {\seed_{d+1}}
	\arrow["{i_1}", no head, from=2-1, to=2-2]
	\arrow["{i_2}", no head, from=2-2, to=2-3]
	\arrow["{i_{d-1}}", no head, from=2-3, to=2-4]
	\arrow["1", no head, from=2-4, to=1-5]
	\arrow["{i_{d}}", no head, from=2-4, to=2-5]
	\arrow["1", no head, from=2-5, to=2-6]
\end{tikzcd}\]
Note that, by the induction hypothesis, $x_{\seed''1}(p) = 0$. 
We now look at the seed $\seed_{d-1}$. 
Note that the rank~$2$ quiver obtained from $Q(\seed_{d-1})$ after freezing all the vertices $[n] - \{1,i_d\}$ is acyclic,  so we may think of~$p$ as being in $\var(\widetilde{\seed}_{d-1})$, where $\widetilde{\seed}_{d-1}$ is the seed obtained by freezing $[n] - \{1,i_d\}$. 
Let $|b_{1i}^{\seed_{d-1}}|$ denote the number of arrows between $1$ and $i$ in $\widetilde{\seed}_{d-1}$. 
Note that if $\zeta$ is a primitive $|b_{1i}^{\seed_{d-1}}|$-root of unity, $(\zeta, 1, \ldots, 1)$ belongs to the stabilizer of $p$ under the action of $\dil(\widetilde{\seed}_{d-1})$. By assumption we have $|b_{1i_d}^{\seed_{d-1}}| \geq 2$, so $p$ has nontrivial stabilizer and is deep.
By Lemma~\ref{lem:zeroes-sticky-rank-2}, it follows that $x_{\seed'1}(p) = 0$.
\end{proof}

\begin{remark}
Proposition~\ref{prop:so-may-deep} can be generalized to algebras where the $|b_{1i}|$ constraint is violated on a small (or well controlled) set. 
The proof shows that $p$ is not within the set of cluster tori connected to $\seed$ by mutations where $1$ remains connected to all other mutable vertices by at least $2$ arrows. 
\end{remark}

\begin{remark}
While the proof of Proposition~\ref{prop:so-may-deep} uses the existence of a stabilizer to establish that $p$ is deep in certain rank $2$ algebras, this is not necessary. 
It is possible (but more tedious) to show $p$ is deep directly from the exchange relations and our assumptions on $\seed$. 
\end{remark}

\section{Mysterious points}\label{sec:mysterious-points}

Proposition~\ref{prop:so-may-deep} allows us to show that many points are deep without any reference to their stabilizer under the action of the cluster dilation group. In this section, we take advantage of this in order to construct several examples of mysterious points in locally acyclic cluster algebras.  In particular, this implies that \cite[Conjecture 1.1]{deep} is false in general.

\begin{proof}[Proof of Theorem~\ref{thm:main2-intro}]
Recall that our quiver is
\[\begin{tikzcd}
	3 && 1 && 2
	\arrow["a", from=1-1, to=1-3] 
	\arrow["b"', from=1-5, to=1-3] 
\end{tikzcd}\]
With $\gcd(a,b)=1$ and $\min(a,b) \geq 2$.

The gcd vector $\vec d(\exmatrix{B})$ is the vector with $\vec d_i = \gcd(b_{i1}, b_{i2}, \ldots, b_{in})$. This vector is a mutation invariant \cite{SevenGCD}. 
Because $\vec d = (1,b,a)$, we must have that $b'_{13} \in a \mathbb Z$ and  $b'_{12} \in b \mathbb Z$ in every seed.
But because $\gcd (b'_{12}, b'_{13}) = 1$ neither can be $0$. 
Thus $\seed$ satisfies the hypothesis of Proposition~\ref{prop:so-may-deep}. 

We note that the locus $\{x_1 = x_1' = 0\}$ is nonempty, given by the equation $1 + x_2^bx_3^a = 0$. 
We claim that any point in this locus is mysterious. 
Any such point is deep by Proposition \ref{prop:so-may-deep}. As observed in Example \ref{ex:dilation-group}, any element $(t_1, t_2, t_3) \in \dil(\seed)$ already satisfies $t_1 = 1$, and this implies that any point in the locus $\{x_1 = x_1' = 0\}$ has trivial stabilizer. Thus, any such point is mysterious.
\end{proof}

\subsection{Keys} \label{subsec:keys}
The class of quivers described in Theorem~\ref{thm:main2-intro} are examples of a more general class of quivers with well understood mutation classes (Definition~\ref{def:keys} below). 
They were described precisely in \cite{Ervin}. 
We will not use the full description, so we state only the necessary definitions and a few corollaries, before giving examples.

\begin{definition}[{\cite[Definition~4.1]{Ervin}}]\label{def:keys}
A \emph{key} is a quiver $Q$ with two distinguished vertices $k,k'$ such that:
\begin{itemize}
    \item $|b_{kk'}| < 2$, and otherwise (for $i \neq j$ and $\{i,j\} \neq \{k,k'\}$) $|b_{ij}| \geq 2$, 
    \item $Q$ is acyclic,
    \item $b_{ki} > 0 \iff b_{k'i} > 0$ for all $i \not \in \{k,k'\}$.
\end{itemize}
\end{definition}

\begin{example}
In Theorem~\ref{thm:main2-intro}, we have a family of keys all with distinguished vertices $\{k,k'\} = \{1,3\}$.
\end{example}

Keys have the smallest number of arrows of any quiver they can be mutated into.
In fact, this is essentially true for each $|b_{ij}|$ (with $i,j$ fixed).

\begin{theorem}[{\cite[Lemmas~4.21 \& 4.23, cf. Corollary 4.24]{Ervin}}]\label{thm:keys-grow}
Let $Q$ be a key with distinguished vertices $\{k,k'\}$ and $\seed = (\exmatrix{B}, \exmatrix{x})$ a seed whose associated quiver has mutable part $Q$.
If $b_{kk'}(Q) = 0$ then $b'_{ij} \geq b_{ij}$ for every seed $(B', \mathbf{z}') \in \seeds(\seed)$.
(If $b_{kk'}(Q) = \pm 1$ then instead $b'_{ij} \geq \min(b_{ij}, b_{\sigma(i)\sigma(j)})$ where $\sigma$ is the transposition $(k k')$.)
\end{theorem}

From Theorem~\ref{thm:keys-grow} and Proposition~\ref{prop:so-may-deep}, we have the following.

\begin{corollary}\label{cor:keys-deep}
Suppose the mutable part of $Q$ is a key with $1 \not \in \{k, k'\}$.
Then the locus $\{x_1=x_1'=0\}$ is nonempty and deep. 
\end{corollary}

\begin{proof}
To see that this locus is nonempty, fix $x_2, \ldots x_{n-1} \neq 0$ and take a solution to the exchange relation 
\[0 = x_1 x_1' = \prod_{b_{i1} > 0
} x_i^{b_{i1}} + \prod_{b_{i1} < 0} x_i^{-b_{i1}}\] 
(as a polynomial in $x_n$).
In order to apply Proposition~\ref{prop:so-may-deep}, we need to know that every $x_i \neq 0$ for $i > 1$ on this locus. 
This follows from the fact that $b_{i1}\neq 0$ and $Q$ is acyclic (see \cite[Corollary 5.4]{muller}).
\end{proof}

\begin{example}
Consider the key 

\begin{center}
\begin{tikzcd}
    & 2 \ar[ddl, "a"'] \ar[ddr, "1"] & \\
    &&\\
    1 \ar[rr, "b"] & & 3
    
\end{tikzcd}
\end{center}
with $\gcd(a,b)=1$. We consider the locus $\{x_1 = x_1' = 0\}$, which is given by $x_2^{a}+x_3^{b} = 0$, $x_3 \neq 0, x_2 \neq 0$ (the nonzero conditions on $x_2, x_3$ come from the exchange relations $x_2x_2' = 1 + x_1^{a}x_3$ and $x_3x_3' = 1 + x_1^{b}x_2$).
Any point in this locus is deep by Corollary~\ref{cor:keys-deep}, and it is easy to see that it has a trivial stabilizer. 
Therefore every point in this locus is mysterious.
\end{example}

\begin{example}
Consider the key 

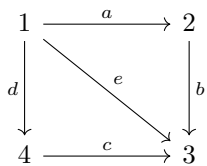
\begin{center}
\begin{tikzcd}
1 \ar[rr, "a"] \ar[dd, "d"'] \ar[ddrr,"e"] & & 2 \ar[dd, "b"] \\
& & \\
4 \ar[rr, "c"] & & 3
\end{tikzcd}
\end{center}
(thus $\min(a,b,c,d,e)\geq 2$). We consider the union of loci  $\{x_1 = x_1' = 0\} \bigsqcup\{x_3 = x_3' = 0\}$. 
Any point in this set is deep by Corollary~\ref{cor:keys-deep}, and it is easy to see that they each have a trivial stabilizer. 
Therefore every point in this set is mysterious.
\end{example}

\subsection{Forks}
The quivers from Definition~\ref{def:keys} are closely related to a family of quivers considered by Warkentin in \cite{warkentinThesis}. These give additional families of quivers with deep (and often mysterious) points. We again summarize the relevant results and give further examples.

\begin{definition}\label{def:abundant}
A quiver is \emph{abundant} if $|b_{ij}| \geq 2$ for all $i \neq j$.
\end{definition}

\begin{definition}[{\cite[Definition 2.1]{warkentinThesis}}]\label{def:fork}
A \emph{fork} $Q$ is an abundant quiver which is not acyclic, with a distinguished vertex $r$ (called the \emph{point of return}) such that:
\begin{enumerate}
    \item the full subquiver of $Q$ with vertices $j$ with $r \rightarrow j$ (resp. $j \rightarrow r$) forms an acyclic subquiver,
    \item for each oriented cycle through $r$, $i \rightarrow r \rightarrow j \rightarrow i$, we have $b_{ij} > \max(b_{ir}, b_{rj}).$
\end{enumerate}
\end{definition}

\begin{definition}
The \emph{fork-less part} of a mutation class is the subset of quivers which are not forks.
\end{definition}

\begin{proposition}[{\cite[Lemma 2.8, Corollary 4.5]{warkentinThesis}}]\label{prop:tree-lem}
If $Q$ is a fork with point of return $r$, then $\mu_j(Q)$ is a fork with point of return $j$ for $j \neq r$.
Thus the fork-less part of a mutation class is connected.
\end{proposition}

Because forks are abundant, any mutation sequence that goes through forks will automatically satisfy the $|b_{1i}|\geq 2$ condition in Proposition~\ref{prop:so-may-deep}.
When the fork-less part is finite, then checking the conditions of Proposition~\ref{prop:so-may-deep} becomes a finite computation.
This is especially easy for abundant acyclic quivers.

\begin{proposition}[{\cite[Lemmas 2.5 and 2.8]{warkentinThesis}}]\label{def:abundance-for-acyclic}
Suppose $Q$ is an abundant acyclic quiver. Then every mutation equivalent quiver is abundant.
\end{proposition}

We thus have the following corollary (whose proof is essentially identical to Corollary~\ref{cor:keys-deep}).

\begin{corollary}\label{cor:abundant-acyclic-deep}
If the mutable part of $Q$ is an abundant acyclic quiver,
then the locus $\{x_1=x_1'=0\}$ is nonempty and deep. 
\end{corollary}

\begin{example}
Consider the family of abundant acyclic quivers $Q$

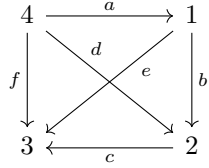
\begin{center}
\begin{tikzcd}
4 \ar[rr, "a"] \ar[dd, "f"'] \ar[ddrr,"d", pos=0.3] & & 1 \ar[dd, "b"] \ar[ddll, "e", pos=0.3] \\
& & \\
3  & & 2 \ar[ll, "c"]
\end{tikzcd}
\end{center}
with $\min(a,b,c,d,e,f) \geq 2$. The locus $\{x_1 = x_1' = 0\}$ is given by $x_4^{a} + x_2^{b}x_3^{e} = 0$ and $x_2x_3x_4 \neq 0$, so it is nonempty. Any point in this locus is deep, and it is mysterious if $\gcd(a,b,e) = 1$. Note that there are choices of $a,b,c,d,e,f$ such that the quiver is really full rank and the locus $\{x_1 = x'_1 = 0\}$ consists entirely of mysterious points.
\end{example}

\begin{example}
Consider the abundant acyclic quiver $Q$

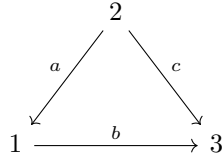
\begin{center}
\begin{tikzcd}
    & 2 \ar[ddl, "a"'] \ar[ddr, "c"] & \\
    &&\\
    1 \ar[rr, "b"] & & 3
\end{tikzcd}
\end{center}
with $\min(a,b,c) \geq 2$.
All points in the union of loci $\{x_1=x_1'=0\} \bigsqcup \{x_2=x_2'=0\} \bigsqcup \{x_3=x_3'=0\} \subset \var(\seed(Q))$ are deep by Corollary~\ref{cor:abundant-acyclic-deep}.
If $\gcd(a,b)=\gcd(a,c)=\gcd(b,c)=1$, then all of these points are also mysterious.
\end{example}

For quivers which are not mutation acyclic, it is more difficult to describe a point in the cluster variety $\var(\seed(Q))$. See \cite[Warning 5.5]{beyer-muller}. Nonetheless, if one can construct a point satisfying the equations in Proposition~\ref{prop:so-may-deep}, then we can show it is deep.

\begin{example}
Consider the abundant quiver $Q$

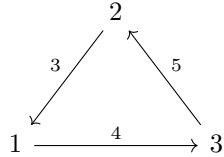
\begin{center}

\begin{tikzcd}
    & 2 \ar[ddl, "3"']& \\
    &&\\
    1 \ar[rr, "4"] & & 3 \ar[uul, "5"']
    
\end{tikzcd}
\end{center}
Then $Q$ is not mutation equivalent to an acyclic quiver (indeed, every quiver mutation equivalent to $Q$ is a fork).
By \cite[Lemma 5.3]{beyer-muller}, any deep point either sends exactly one cluster variable in each seed to $0$, or sends all of them to $0$.
So all (not constant $0$) points in the union of loci 
\[\{x_1=x_1'=0\} \bigsqcup \{x_2=x_2'=0\} \bigsqcup \{x_3=x_3'=0\} \subset \var(\seed(Q))\] are deep by Proposition~\ref{prop:so-may-deep}, and can quickly be shown to be mysterious.
\end{example}

\subsection{Locally acyclic}
We have seen several examples of acyclic quivers which have mysterious points.
So it is unsurprising that there are quivers which are not acyclic, but are locally-acyclic, which have mysterious points.

\begin{example}[{\cite[Example 12.30]{warkentinThesis}}]
Consider a quiver $Q$ of the form: 
\begin{center}
\begin{tikzcd}
    & 2 \ar[ddl, "a"'] & \\
    &4 \ar[u, "d", pos=0.4] \ar[dl, "e", pos=0.4] \ar[dr, "f"', pos=0.4] &\\
    1 \ar[rr, "b"] & & 3 \ar[uul, "c+ab"']
    
\end{tikzcd}
\end{center}
with $\min(a,b,c,d,e,f)\geq 2$.
This quiver is locally acyclic (and in fact Banff).
The subquiver supported by the vertices $\{1,2,3\}$ is mutation acyclic; indeed it is acyclic in $\mu_1(Q)$.
Thus we can specify a point in the cluster algebra $\var(\seed(Q))$ by specifying a nonzero value for $x_4$ and the cluster variables in each of $\seed(Q), \mu_1(\seed(Q)), \mu_2(\mu_1(\seed(Q))),$ and $\mu_3(\mu_1(\seed(Q)))$ (satisfying the exchange relations).
If we choose $x_1=0$ in both $\seed(Q)$ and $\mu_1(\seed(Q))$, then the other cluster variables will necessarily be nonzero, and by Proposition~\ref{prop:so-may-deep} these points will be deep. 
The same argument shows that all the points in 
$$\{x_1 = x'_1 = 0\} \bigsqcup \{ x_2 = x'_2 = 0\} \bigsqcup \{ x_3 =x'_3 = 0\}$$ are deep (where $x_i$ is the cluster variable in $\mu_1(\seed(Q))$ and $x_i'$ is the cluster variable in $\mu_i(\mu_1(\seed(Q)))$).

The quiver $Q$ can be shown locally acyclic in multiple ways. In particular, $\mu_1(Q)$ has a sink vertex at $2$ and the subquiver supported by the vertices $\{1,3,4\}$ is acyclic in $Q$. 
So the points in the cluster variety with $x_4=0$ in both $\seed(Q)$ and $\mu_4(\seed(Q))$ are also deep. 

Often these points have no stabilizer. For example, when the values $a,b,c,d,e,f$ are all pairwise coprime.
So this gives many examples of mysterious points.
\end{example}

Further examples of locally-acyclic quivers with finite fork-less part (which may consist of only abundant quivers) can be constructed via \cite[Theorem 4.27]{MCR}.

We also remind the reader of \cite[Propositions~3.15-16]{deep}, which state that having a mysterious point is a quasi-cluster isomorphism invariant.

\begin{proposition}\label{prop:adding-frozen-variables}
If $Q$ is a quiver so that its cluster algebra has mysterious points, then adding frozen variables to $Q$ will still produce a cluster algebra with mysterious points. 
\end{proposition}

\begin{proof}
This is the contrapositive of \cite[Proposition 3.18]{deep}
\end{proof}

Thus, while some of the examples of cluster algebras with mysterious points we have produced may not have really full rank, we can always add frozen variables to produce example of cluster algebras with really full rank and mysterious points.

\subsection{Abundant without mysterious points}
Based on the previous examples, it seems that the property of having mysterious points is subtle and depends on number-theoretic properties of the multiplicities of arrows. In fact, we have the following result. 

\begin{proposition}\label{prop:3-vertices}
Consider the cluster algebra $A$ with exchange matrix associated to the quiver
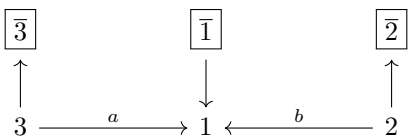
\[\begin{tikzcd}
	{\boxed{\overline{3}}} && {\boxed{\overline{1}}} && {\boxed{\overline{2}}} \\
	3 && 1 && 2
	\arrow[from=1-3, to=2-3]
	\arrow[from=2-1, to=1-1]
	\arrow["a", from=2-1, to=2-3]
	\arrow[from=2-5, to=1-5]
	\arrow["b"', from=2-5, to=2-3]
\end{tikzcd}\]
\begin{enumerate}
    \item If either $\min(a,b) = 1$ or $\gcd(a,b) > 1$, then $A$ has no mysterious points.
    \item Else, $A$ has mysterious points. 
\end{enumerate}
\end{proposition}
\begin{proof}
We assume that $a$ and $b$ are not both equal to $1$; otherwise the result follows from Theorem \ref{thm:main-notintro}. The cluster dilation group is
\[
\dil(\seed) = \{t = (t_1, t_2, t_3, t_{\overline{1}}, t_{\overline{2}}, t_{\overline{3}}) \in  (\C^{\times})^{6} \mid t_1^{a}t_{\overline{3}} = t_3^{a}t_2^{b}t_{\overline{1}} = t_{1}^{b}t_{\overline{2}} = 1\}.
\]
Let $p \in \var$ have trivial stabilizer in $\dil(\seed)$. If $p \in \mathcal{O}_{\emptyset} = T(\seed)$ we are done. If $p \in \mathcal{O}_{\{3\}}$ then $x_1(p)x_2(p)\neq 0$, so we may freeze vertex $2$ in order to reduce to the case of Lemma \ref{lem:rank2}. Similarly if $p \in \mathcal{O}_{\{1\}}$. 

It remains to treat the case of the maximal independent sets $\{2,3\}$ and $\{1\}$. If $p \in \mathcal{O}_{\{2,3\}}$ but either $x'_2(p)$ or $x_3'(p) \neq 0$, we mutate at the corresponding vertex to reduce to the case of the previous paragraph. So we assume $x_2(p) = x'_2(p) = x_3(p) = x_3'(p) = 0$. If, say, $a > 1$, this implies that $(\zeta, 1, 1, 1, 1, 1) \in \Stab_{\dil(\seed)}(p)$, where $\zeta$ is a primitive $a$-th root of unity. So this case cannot happen.

Now assume that $p \in \mathcal{O}_{\{1\}}$. We may also assume that $x'_1(p) = 0$. If $\gcd(a,b) > 1$, we have $(1, \xi, 1, 1, 1, 1) \in \Stab_{\dil(\seed)}(p)$, where $\xi$ is a primitive $\gcd(a,b)$-root of unity, a contradiction. If $\min\{a,b\} = 1$, say $a = 1$, then we apply a strategy similar to that of the proof of Case 2.1 of Theorem \ref{thm:main-notintro}, and we see that $p \in T(\mu_1\mu_3(\seed))$, indeed
\[
x''_1 = \frac{x'_3 + x_2^{b}x_{\overline{1}}}{x_1} = 
\frac{\frac{1 + x_{\overline{3}}x_1}{x_3}+ x_2^{b}x_{\overline{1}}}{x_1} = 
\frac{1 + x_{\overline{3}}x_1 + x_{3}x_2^{b}x_{\overline{1}}}{x_1x_3} = 
\frac{x'_1 + x_{\overline{3}}}{x_3}
\]
so $x''_1(p) \neq 0$ and thus $p$ is not mysterious. We are done with (1). Statement (2) is Theorem \ref{thm:main2-intro}. 
\end{proof}

\bibliographystyle{plain}
\bibliography{bibliography.bib}
\end{document}